\documentclass[11pt, amsfonts]{amsart}

\usepackage{amsmath,amssymb,amsthm,color,enumerate}
\usepackage[all]{xy}
\usepackage[usenames]{xcolor}
\definecolor{refkey}{rgb}{0.6, 0.7, 0.4}
\definecolor{labelkey}{rgb}{0, 0.7, 0.5}
\usepackage{todonotes}
\SelectTips{eu}{12}

\newtheorem{theorem}{Theorem}[section]
\newtheorem{proposition}[theorem]{Proposition}
\newtheorem{lemma}[theorem]{Lemma}
\newtheorem{corollary}[theorem]{Corollary}

\theoremstyle{definition}
\newtheorem{definition}[theorem]{Definition}

\theoremstyle{remark}

\numberwithin{equation}{section}

\newcommand{\ZZ}{\mathbb{Z}}
\newcommand{\st}{\mathrm{star}}
\newcommand{\lk}{\mathrm{link}}

\makeatletter
\@namedef{subjclassname@2020}{%
  \textup{2020} Mathematics Subject Classification}
\makeatother

\title{Independence complexes of $(n \times 6)$-grid graphs}

\author[T. Matsushita]{Takahiro Matsushita}
\address{Department of Mathematical Sciences, University of the Ryukyus, Nishihara-cho, Okinawa 903-0213, Japan}
\email{mtst@sci.u-ryukyu.ac.jp}

\author[S. Wakatsuki]{Shun Wakatsuki}
\address{Graduate School of Mathematics, Nagoya University, Furocho, Chikusaku,
Nagoya, 464-8602, Japan}
\email{shun.wakatsuki@math.nagoya-u.ac.jp}


\begin{document}

\baselineskip.525cm

\maketitle

\begin{abstract}
We determine the homotopy types of the independence complexes of the $(n \times 6)$-square grid graphs. In fact, we show that these complexes are homotopy equivalent to wedges of spheres.
\end{abstract}

\section{Introduction}

For a finite simple graph $G = (V,E)$, a subset $\sigma$ of the vertex set $V$ of $G$ is called \emph{independent} if no two elements in $\sigma$ are adjacent. The family of independent sets of $G$ forms a simplicial complex $I(G)$, which is called the \emph{independence complex of $G$}. Homotopy types of independence complexes and their connection with combinatorial properties of graphs have been extensively studied in the last two decades \cite{Kozlov book}.

%Independence complexes of graphs have been extensively studied in the last two decades from topological and combinatorial viewpoints.

The independence complex $I(G)$ of $G$ is the simplicial complex whose vertex set is $V$ and whose set of minimal non-faces is the set $E$ of edges in $G$. Hence the independence complex is another formulation of the clique complex, which has appeared in several branches in mathematics. The Vietoris--Rips complex, which appears in topological data analysis and geometric group theory \cite{PRSZ}, and the order complex of a poset \cite{Kozlov book} are typical examples of clique complexes. In particular, the barycentric subdivision of every simplicial complex is a clique complex, and hence a very wide class of geometric objects can be obtained from the independence complex. This means that the independence complex is a fundamental construction of simplicial complexes.

However, it is in general quite difficult to determine the homotopy type of independence complex $I(G)$ even if the graph $G$ can be easily described. It has been studied by several authors to determine the homotopy types or homotopy invariants of particular classes of graphs (\cite{Barmak}, \cite{Braun}, \cite{BH}, \cite{EH}, \cite{Engstrom}, \cite{GSS}, \cite{Iriye}, \cite{Kawamura}, \cite{Kozlov}, \cite{MT}, \cite{Matsushita1}, \cite{Matsushita2}).
%The homotopy types of independence complexes have been studied by several authors (see \cite{Adamaszek1}, \cite{Barmak}, \cite{EH}, \cite{Engstrom}, \cite{Iriye}, \cite{Jonsson1}, \cite{Jonsson2} for example).
In this paper, we treat the independence complexes of certain square grid graphs $\Gamma_{n,k}$ defined as follows: Let $n$ and $k$ be positive integers. We define the graph $\Gamma_{n,k}$ by
\[ V(\Gamma_{n,k}) = \{ (x,y) \in \ZZ^2\; | \; 1 \le x \le n, \; 1 \le y \le k\},\]
\[ E(\Gamma_{n,k}) = \{ \{ (x,y), (x', y')\} \; | \; (x,y), (x', y') \in V(\Gamma_{n,k}), \; |x' - x| + |y' - y| = 1\}.\]
The goal in this paper is to determine the homotopy type of $I(\Gamma_{n,6})$, for all $n$.

Before giving the precise statement of our main result, we review the background and known results in the independence complex of square grid graphs. The graph $\Gamma_{n,1}$ is the path graph $P_n$ of $n$ vertices, and the homotopy type of $I(\Gamma_{n,1})$ was determined by Kozlov \cite{Kozlov}. The homotopy types of $I(\Gamma_{n,2})$ and $I(\Gamma_{n,3})$ were determined by Adamaszek \cite{Adamaszek2}. The Euler characteristic of $I(\Gamma_{n,4})$ was determined by Okura \cite{Okura}, and the homotopy types of $I(\Gamma_{n,4})$ and $I(\Gamma_{n,5})$ were recently determined by the authors \cite{MW}.

Our grid graph $\Gamma_{n,k}$ is the cartesian product $P_n \times P_k$ of two path graphs $P_n$ and $P_k$. As related graphs, the independence complexes of $C_n \times C_k$ and $P_n \times C_k$ have also been studied. Here $C_n$ denotes the cycle graph with $n$ vertices.
%The study of the independence complexes of these graphs was motivated from statistical physics.
Fendley, Schoutens, and van Eersten \cite{FSV} suggested several conjectures related to the Euler characteristic of $I(C_n \times C_k)$ from a viewpoint of statistical physics. Jonsson \cite{Jonsson1} solved one of their conjectures, which states that the reduced Euler characteristic of $I(C_n \times C_k)$ is $1$ when $n$ and $k$ are coprime. After that, Bousquet-M\'elou, Linusson, and Nevo \cite{BLN} studied the homotopy types of several families of independence complexes of grid graphs, and since then the independence complexes of $C_n \times C_k$, $P_n \times C_k$, and $P_n \times P_k$ have been studied by several authors (\cite{Adamaszek2}, \cite{Iriye}, \cite{Jonsson2}, \cite{Jonsson3}, \cite{MW}, \cite{Okura}, \cite{Thapper}).

Now we state our main result in this paper:

\begin{table}[t]
  \begin{tabular}{cccccccc}
    $k$   & 0 & 1 & 2 & 3 & 4 & 5 & 6 \\
    \hline
    $\nu$ & 0 & 0 & 0 & 2 & 2 & 4 & 4 \\
    $\mu$ & 2 & 2 & 4 & 4 & - & - & -
  \end{tabular}
  \caption{} \label{table:nuMuValues}
\end{table}

\begin{theorem} \label{main theorem}
  For $n\geq 1$, the homotopy type of $I(\Gamma_{n,6})$ is given as follows.
  \begin{enumerate}[$(1)$]
    \item Assume that $n$ is odd and write $n=14m+2k+1$ with $m\geq 0$ and $0\leq k\leq 6$.
      Then there is a homotopy equivalence:
      \[
      I(\Gamma_{n,6})\simeq
      S^{n'}
      \vee\Bigl(\bigvee_{n'-m\leq i\leq n'-1}\bigvee_6S^i\Bigr)
      \vee\bigvee_\nu S^{n'-m-1},
      \]
      where $n'=21m+3k+1$ and $\nu$ is a number defined by Table \ref{table:nuMuValues}.
    \item Assume that $n$ is even and write $n=14m+2k$ with $m\geq 0$ and $0\leq k\leq 6$.
      If $n\leq 6$, we have
      \[
      I(\Gamma_{2,6}) \simeq S^2, \;
      I(\Gamma_{4,6}) \simeq \bigvee_3S^5, \;
      I(\Gamma_{6,6}) \simeq \bigvee_3S^8.\]
      If $n>6$, there is a homotopy equivalence:
      \[
      I(\Gamma_{n,6}) \simeq
      \begin{cases}
        \displaystyle
        \bigvee_5S^{n'}
        \vee\Bigl(\bigvee_{n'-m< i\leq n'-1}\bigvee_6S^i\Bigr)
        \vee\bigvee_\mu S^{n'-m} & (0\leq k\leq 3, m\geq 1) \\
        \displaystyle
        \bigvee_5S^{n'}
        \vee\Bigl(\bigvee_{n'-m\leq i\leq n'-1}\bigvee_6S^i\Bigr) & (4\leq k\leq 6),
      \end{cases}
      \]
      where $n'=21m+3k-1$ and $\mu$ is a number defined by Table \ref{table:nuMuValues}.
  \end{enumerate}
\end{theorem}

Finally, we discuss the Euler characteristic of $I(\Gamma_{n,k})$. For a positive integer $k$, let $f_k$ denote the function assigning $\chi(I(\Gamma_{n,k}))$ to $n$. The known results (\cite{Kozlov}, \cite{Adamaszek2}, \cite{MW}) showed that $f_1$, $f_2$, $f_3$, $f_5$ are periodic functions with periods $6$, $4$, $8$, and $40$, respectively. On the other hand, by the computation of Okura \cite{Okura}, $f_4$ is not a bounded function and $|f_4(n)|$ tends to infinity.
%As was showed in \cite{Jonsson2}, the Euler characteristic of $I(P_n \times C_k)$ with odd $k$ is quite simple and has a small period.
Our main result implies the following:

\begin{corollary}
  The function $f_6$ assigning $\chi(I(\Gamma_{n,6}))$ to $n$ is a periodic function with period $28$,
  whose values are shown in Table \ref{table:f6Values}.
\end{corollary}

\begin{table}[t]
  \begin{tabular}{ cccccccccccccccccccccccccccc }
    \(n\)      & 1& 2& 3& 4& 5& 6& 7& 8& 9&10&11&12&13&14 \\
    \(f_6(n)\) & 0& 2& 2&-2& 0& 4& 0&-4& 2& 6&-2&-4& 4& 4 \\[10pt]
    \(n\)      &15&16&17&18&19&20&21&22&23&24&25&26&27&28 \\
    \(f_6(n)\) &-4&-2& 6& 2&-4& 0& 4& 0&-2& 2& 2& 0& 0& 0
  \end{tabular}
  \caption{} \label{table:f6Values}
\end{table}

%Hence $f_1, f_2, f_3, f_5,$ and $f_6$ are bounded, and $f_4$ is unbounded. Our computer computation suggests that $f_9$ and $f_{11}$ are bounded and $f_7$, $f_8$, and $f_{10}$ are unbounded. Although there is a simple rule that the Euler characteristics of $I(C_n \times C_k)$ and $I(P_n \times C_k)$ satisfy, there seems no simple rules in the case of $I(P_n \times P_k)$.

The rest of this paper is organized as follows: In Section 2, we review several definitions and facts related to independence complex. Section 3 is devoted to the proof of Theorem \ref{main theorem}.

\section{Preliminaries}

In this section, we review definitions and facts related to independence complexes which we need in this paper. Our main tools are the cofiber sequence (Theorem \ref{thm cofiber}) and the fold lemma (Corollary \ref{cor fold lemma}). For a comprehensive introduction to simplicial complexes, we refer to \cite{Kozlov book}.

A \emph{(finite simple) graph} is a pair $(V,E)$ such that $V$ is a finite set and $E$ is a subset of the family $\binom{V}{2}$ of $2$-element subsets of $V$. A subset $\sigma$ of $V$ is \emph{independent} or \emph{stable} if there are no $v,w \in \sigma$ such that $\{ v,w \} \in E$. The \emph{independence complex of a graph $G$} is the simplicial complex whose vertex set is $V$ and whose simplices are independent sets in $G$, and is denoted by $I(G)$.

A \emph{subgraph of $G = (V,E)$} is a graph $G' = (V', E')$ such that $V' \subset V$ and $E' \subset E$. The subgraph $G'$ of $G$ is said to be \emph{induced} if $E' = \binom{V'}{2} \cap E$. For a subset $S$ of $V$, we write $G - S$ to mean the induced subgraph of $G$ whose vertex set is $V - S$. For a vertex $v$ of $G$, we write $G - v$ instead of $G - \{ v\}$. Note that a subgraph $G'$ of $G$ is induced if and only if $I(G')$ is a subcomplex of $I(G)$.

If $G$ is a disjoint union $G_1 \sqcup G_2$ of two subgraphs $G_1$ and $G_2$, then $I(G)$ coincides with the join $I(G_1) * I(G_2)$. In particular, if $G$ has an isolated vertex, then $I(G)$ is contractible. Let $K_2$ denote the complete graph with $2$ vertices, \emph{i.e.} $K_2$ is the graph consisting of two vertices and one edge. Since $I(K_2) = S^0$, we have $I(K_2 \sqcup G) \simeq \Sigma I(G)$. Here, $\Sigma$ denotes a suspension.

Let $K$ be an (abstract) simplicial complex and $S$ a subset of $V(K)$. Let $K - S$ denote the simplicial complex consisting of the simplices of $K$ disjoint from $S$. For a face $\sigma$ of $K$, define the \emph{star} $\st(\sigma)$ and the \emph{link} $\lk(\sigma)$ by
\[ \st(\sigma) = \{ \tau \in K \; | \; \sigma \cup \tau \in K \},\; \lk(\sigma) = \st(\sigma) - \sigma.\]
For $v \in V$, we write $\st(v)$ and $\lk(v)$ instead of $\st(\{ v\})$ and $\lk (\{ v\})$, respectively.

There is a functor called geometric realization from the category of simplicial complexes to the category of topological spaces \cite{Kozlov book}. We often identify the simplicial complex $K$ with its geometric realization $|K|$, and apply topological terms to simplicial complexes.

\begin{lemma}
  Let $K$ be a simplicial complex and $v$ a vertex of $K$.
  Then $K$ is homeomorphic to the mapping cone of the inclusion $\lk(v) \hookrightarrow K - v$,
  and hence there is a cofiber sequence
  \[\lk(v)\to K - v \to K.\]
  In particular, if the inclusion $\lk (v) \hookrightarrow K - v$ is null-homotopic,
  then
\[ K \simeq (K - v) \vee \Sigma  \lk (v).\]
\end{lemma}

Let $G$ be a graph and let $v \in V(G)$. Let $N(v)$ denote the set of vertices of $G$ adjacent to $v$, and set $N[v] = N(v) \cup \{ v\}$. It is straightforward to see $I(G) - v = I(G - v)$ and $\lk(v) = I(G - N[v])$.
Applying the above lemma to $K=I(G)$, we have the following:

\begin{theorem}[see \cite{Adamaszek1}] \label{thm cofiber}
Let $G$ be a graph and $v$ a vertex of $G$. Then there is a cofiber sequence
\[ I(G - N[v]) \to I(G - v) \to I(G).\]
In particular, if the inclusion $I(G - N[v]) \to I(G - v)$ is null-homotopic, then
\[ I(G) \simeq I(G - v) \vee \Sigma I(G - N[v]).\]
\end{theorem}

This theorem implies the following simple argument, which will be frequently used in the subsequent section.

\begin{corollary}[fold lemma \cite{Engstrom}] \label{cor fold lemma}
Let $G$ be a graph, and $v$, $w$ vertices of $G$. Assume that $v \ne w$ and $N(v) \subset N(w)$. Then the inclusion $I(G - w) \hookrightarrow I(G)$ is a homotopy equivalence.
\end{corollary}
%\begin{proof}
%We write the proof for the reader's convenience. By Theorem \ref{thm cofiber}, it suffices to see that $v$ is an isolated vertex of $G - N[w]$. If $v \in N[w]$, then $v \ne w$ means that $v \in N(w)$. This means $w \in N(v) \subset N(w)$. This is a contradiction. Hence $v \in V(G - N[w])$. Since $N(v) \subset N(w)$, we have that $v$ is an isolated vertex in $G - N[w]$. This completes the proof.
%\end{proof}

\section{Proofs}

In this section, we write $\Gamma_n$ instead of $\Gamma_{n,6}$. The goal of this section is to prove Theorem \ref{main theorem}, which determine the homotopy type of $I(\Gamma_n)$.

\begin{figure}[t]
\begin{picture}(300,140)(0,-20)
\multiput(30,0)(0,20){6}{\circle*{3}}
\multiput(50,0)(0,20){6}{\circle*{3}}
\multiput(70,0)(0,20){6}{\circle*{3}}
\multiput(90,0)(0,20){6}{\circle*{3}}
\multiput(110,20)(0,40){3}{\circle*{3}}

\multiput(30,0)(20,0){4}{\line(0,1){100}}

\multiput(20,20)(0,40){3}{\line(1,0){90}}
\multiput(20,0)(0,40){3}{\line(1,0){70}}

\put(55,-20){$X_n$}
\put(0,47){$\cdots$}

\multiput(190,0)(0,20){6}{\circle*{3}}
\multiput(210,0)(0,20){6}{\circle*{3}}
\multiput(230,0)(0,20){6}{\circle*{3}}
\multiput(250,0)(0,20){6}{\circle*{3}}
\multiput(270,0)(0,20){2}{\circle*{3}}
\multiput(270,80)(0,20){2}{\circle*{3}}

\multiput(190,0)(20,0){4}{\line(0,1){100}}
\multiput(270,0)(0,80){2}{\line(0,1){20}}
\multiput(180,0)(0,20){2}{\line(1,0){90}}
\multiput(180,80)(0,20){2}{\line(1,0){90}}
\multiput(180,40)(0,20){2}{\line(1,0){70}}

\put(215,-20){$Y_n$}
\put(160,47){$\cdots$}
\end{picture}
\caption{} \label{figure XY}
\end{figure}

\subsection{Subgraphs $X_n$, $Y_n$, $A_n$ and $B_n$}
In this subsection, we determine the homotopy types of independence complexes of several induced subgraphs $X_n$, $Y_n$, $A_n$ and $B_n$ of $\Gamma_n$, and relates them to the homotopy type of the independence complex $I(\Gamma_n)$ of $\Gamma_n$.

\begin{definition} \label{definition subgraphs}
For $n \ge 1$, define the induced subgraphs $X_n$, $Y_n$, $A_n$ and $B_n$ of $\Gamma_n$ by
\[ X_n = \Gamma_n - \{ (n,1), (n,3), (n,5)\},\]
\[ Y_n = \Gamma_n - \{ (n,3), (n,4)\},\]
\[ A_n = \Gamma_n - \{ (n,1), (n,5)\},\]
\[ B_n = \Gamma_n - \{ (n,4)\}.\]
See Figures \ref{figure XY} and \ref{figure AB}.
\end{definition}

\begin{figure}[b]
\begin{picture}(300,140)(0,-20)
\multiput(30,0)(0,20){6}{\circle*{3}}
\multiput(50,0)(0,20){6}{\circle*{3}}
\multiput(70,0)(0,20){6}{\circle*{3}}
\multiput(90,0)(0,20){6}{\circle*{3}}
\multiput(110,20)(0,20){3}{\circle*{3}}
\put(110,100){\circle*{3}}

\multiput(30,0)(20,0){4}{\line(0,1){100}}
\put(110,20){\line(0,1){40}}
\multiput(20,20)(0,20){3}{\line(1,0){90}}
\multiput(20,0)(0,80){2}{\line(1,0){70}}
\put(20,100){\line(1,0){90}}

\put(55,-20){$A_n$}
\put(0,47){$\cdots$}
\put(115,37){$v_n$}

\multiput(190,0)(0,20){6}{\circle*{3}}
\multiput(210,0)(0,20){6}{\circle*{3}}
\multiput(230,0)(0,20){6}{\circle*{3}}
\multiput(250,0)(0,20){6}{\circle*{3}}
\multiput(270,0)(0,20){3}{\circle*{3}}
\multiput(270,80)(0,20){2}{\circle*{3}}

\multiput(190,0)(20,0){4}{\line(0,1){100}}
\put(270,0){\line(0,1){40}}
\put(270,80){\line(0,1){20}}
\multiput(180,0)(0,20){3}{\line(1,0){90}}
\multiput(180,80)(0,20){2}{\line(1,0){90}}
\put(180,60){\line(1,0){70}}

\put(215,-20){$B_n$}
\put(160,47){$\cdots$}
\put(275,37){$v_n$}
\end{picture}
\caption{} \label{figure AB}
\end{figure}

The homotopy type of $I(\Gamma_n)$ is determined by Lemmas \ref{lemma X} and \ref{lemma Y}, Corollaries \ref{corollary key} and \ref{corollary small A}, and Proposition \ref{proposition key 2}. Lemmas \ref{lemma X} and \ref{lemma Y} are easy consequences of the fold lemma (Corollary \ref{cor fold lemma}). The others are mainly deduced by Proposition \ref{proposition key}. In this subsection, we prove these propositions, and the determination of the homotopy type of $I(\Gamma_n)$ is postponed to the next subsection.

We first prove Lemmas \ref{lemma X} and \ref{lemma Y}, which determine the homotopy types of $I(X_n)$ and $I(Y_n)$. Here we write ${\rm pt}$ to mean the topological space consisting of one point.

\begin{lemma} \label{lemma X}
For $n\ge 1$,
there is a following homotopy equivalence:
\[ I(X_n) \simeq \begin{cases}
{\rm pt} & (n = 2k+1) \\
S^{3k-1} & (n = 2k).
\end{cases}\]
\end{lemma}
\begin{proof}
The fold lemma implies
\[ I(X_n) \simeq \Sigma^3 I(X_{n-2})\]
for $n \ge 3$ (see Figure \ref{figure X}). The fold lemma again implies
\[ I(X_1) \simeq {\rm pt} \quad \textrm{and} \quad I(X_2) \simeq S^2.\]
This completes the proof.
\end{proof}

\begin{figure}[t]
\begin{picture}(300,140)(0,-20)
\multiput(30,0)(0,20){6}{\circle*{3}}
\multiput(50,0)(0,20){6}{\circle*{3}}
\multiput(70,0)(0,20){6}{\circle*{3}}
\multiput(90,0)(0,20){6}{\circle*{3}}
\multiput(110,20)(0,40){3}{\circle*{3}}

\multiput(30,0)(20,0){4}{\line(0,1){100}}

\multiput(20,20)(0,40){3}{\line(1,0){90}}
\multiput(20,0)(0,40){3}{\line(1,0){70}}

\put(45,-20){$I(X_n)$}
\put(0,47){$\cdots$}

\multiput(190,0)(0,20){6}{\circle*{3}}
\multiput(210,0)(0,20){6}{\circle*{3}}
\multiput(230,0)(0,40){3}{\circle*{3}}
\multiput(230,20)(0,40){3}{\circle{3}}
\multiput(250,0)(0,40){3}{\circle{3}}
\multiput(250,20)(0,40){3}{\circle*{3}}
\multiput(270,20)(0,40){3}{\circle*{3}}

\multiput(190,0)(20,0){2}{\line(0,1){100}}
\multiput(180,0)(0,40){3}{\line(1,0){50}}
\multiput(180,20)(0,40){3}{\line(1,0){30}}
\multiput(250,20)(0,40){3}{\line(1,0){20}}

\put(200,-20){$\Sigma^3 I(X_{n-2})$}
\put(160,47){$\cdots$}
\put(140,47){$\simeq$}
\end{picture}
\caption{} \label{figure X}
\end{figure}

\begin{figure}[b]
\begin{picture}(300,140)(0,-20)
\multiput(30,0)(0,20){6}{\circle*{3}}
\multiput(50,0)(0,20){6}{\circle*{3}}
\multiput(70,0)(0,20){6}{\circle*{3}}
\multiput(90,0)(0,20){6}{\circle*{3}}
\multiput(110,0)(0,20){2}{\circle*{3}}
\multiput(110,80)(0,20){2}{\circle*{3}}

\multiput(20,0)(0,20){2}{\line(1,0){90}}
\multiput(20,80)(0,20){2}{\line(1,0){90}}
\multiput(20,40)(0,20){2}{\line(1,0){70}}
\multiput(30,0)(20,0){4}{\line(0,1){100}}
\multiput(110,0)(0,80){2}{\line(0,1){20}}

\put(45,-20){$I(Y_n)$}
\put(0,47){$\cdots$}

\multiput(190,0)(0,20){6}{\circle*{3}}
\multiput(210,0)(0,20){6}{\circle*{3}}
\multiput(230,0)(0,20){2}{\circle*{3}}
\multiput(230,40)(0,20){2}{\circle{3}}
\multiput(230,80)(0,20){2}{\circle*{3}}
\multiput(250,0)(0,20){2}{\circle{3}}
\multiput(250,40)(0,20){2}{\circle*{3}}
\multiput(250,80)(0,20){2}{\circle{3}}
\multiput(270,0)(0,20){2}{\circle*{3}}
\multiput(270,80)(0,20){2}{\circle*{3}}

\multiput(180,0)(0,20){2}{\line(1,0){50}}
\multiput(180,80)(0,20){2}{\line(1,0){50}}
\multiput(180,40)(0,20){2}{\line(1,0){30}}
\multiput(190,0)(20,0){2}{\line(0,1){100}}
\multiput(230,0)(0,80){2}{\line(0,1){20}}
\put(250,40){\line(0,1){20}}
\multiput(270,0)(0,80){2}{\line(0,1){20}}

\put(200,-20){$\Sigma^3 I(Y_{n-2})$}
\put(160,47){$\cdots$}
\put(140,47){$\simeq$}
\end{picture}
\caption{} \label{figure Y}
\end{figure}

\begin{lemma} \label{lemma Y}
For $n\ge 1$,
there is a following homotopy equivalence:
\[ I(Y_n) \simeq \begin{cases}
S^{3k+1} & (n = 2k+1) \\
S^{3k-1} & (n = 2k).
\end{cases} \]
\end{lemma}
\begin{proof}
The fold lemma implies
\[ I(Y_n) \simeq \Sigma^3 I(Y_{n-2})\]
for $n \ge 3$ (see Figure \ref{figure Y}). The fold lemma again implies
\[ I(Y_1) \simeq S^1 \quad \textrm{and} \quad I(Y_2) \simeq S^2.\]
This completes the proof.
\end{proof}

Set $v_n = (n, 3)$. Note that $A_n - v_n = X_n$ and $B_n - v_n = Y_n$ (see Definition \ref{definition subgraphs}).

\begin{lemma} \label{lemma A - B}
There are following homotopy equivalences:
\begin{enumerate}[$(1)$]
\item For $n \ge 4$, there is a following homotopy equivalence:
\[ I(A_n - N[v_n]) \simeq \Sigma^3 I(B_{n-3}).\]
\item For $n \ge 5$, there is a following homotopy equivalence:
\[ I(B_n - N[v_n]) \simeq \Sigma^5 I(A_{n-4}).\]
\end{enumerate}
\end{lemma}
\begin{proof}
These homotopy equivalences are immediately deduced by the fold lemma. See Figures \ref{figure A - N[v]} and \ref{figure B - N[v]}.
\end{proof}

\begin{lemma} \label{lemma small AB}
The following hold:
\[ I(A_1) \simeq {\rm pt}, I(A_2) \simeq S^2, I(A_3) \simeq S^3,\]
\[ I(B_1) \simeq S^1, I(B_2) \simeq S^2, I(B_3) \simeq S^4, I(B_4) \simeq S^5 \vee S^5.\]
\end{lemma}
\begin{proof}
These homotopy equivalences clearly follow from the fold lemma except for $I(B_4) \simeq S^5 \vee S^5$. This last homotopy equivalence is deduced from Theorem \ref{thm cofiber},
\[ I(B_4 - v_4) \simeq S^5, \quad \textrm{and} \quad I(B_4 - N[v_4]) \simeq S^4.\]
These two homotopy equivalences are deduced from the fold lemma.
\end{proof}

\begin{figure}[t]
\begin{picture}(300,140)(0,-20)
\multiput(30,0)(0,20){6}{\circle*{3}}
\multiput(50,0)(0,20){6}{\circle*{3}}
\multiput(70,0)(0,20){6}{\circle*{3}}
\multiput(90,0)(0,20){2}{\circle*{3}}
\multiput(90,60)(0,20){3}{\circle*{3}}
\put(110,100){\circle*{3}}
\put(90,40){\circle{3}}
\multiput(110,20)(0,20){3}{\circle{3}}

\multiput(30,0)(20,0){3}{\line(0,1){100}}
\multiput(20,0)(0,20){2}{\line(1,0){70}}
\multiput(20,60)(0,20){2}{\line(1,0){70}}
\put(20,40){\line(1,0){50}}
\put(20,100){\line(1,0){90}}
\put(90,0){\line(0,1){20}}
\put(90,60){\line(0,1){40}}

\put(30,-20){$I(A_n - N[v_n])$}
\put(0,47){$\cdots$}
\put(115,37){$v_n$}

\multiput(190,0)(0,20){6}{\circle*{3}}
\multiput(210,0)(0,20){3}{\circle*{3}}
\multiput(210,80)(0,20){2}{\circle*{3}}
\multiput(250,0)(0,20){2}{\circle*{3}}
\multiput(230,60)(20,0){2}{\circle*{3}}
\multiput(250,100)(20,0){2}{\circle*{3}}

\put(210,60){\circle{3}}
\multiput(230,0)(0,20){3}{\circle{3}}
\multiput(230,80)(0,20){2}{\circle{3}}
\multiput(180,0)(0,20){3}{\line(1,0){30}}
\multiput(180,80)(0,20){2}{\line(1,0){30}}

\put(250,80){\circle{3}}

\put(180,60){\line(1,0){10}}
\put(190,0){\line(0,1){100}}
\put(210,0){\line(0,1){40}}
\put(210,80){\line(0,1){20}}
\put(230,60){\line(1,0){20}}
\put(250,0){\line(0,1){20}}
\put(250,100){\line(1,0){20}}

\put(215,-20){$\Sigma^3 I(B_{n-3})$}
\put(160,47){$\cdots$}
\put(140,47){$\simeq$}
\end{picture}
\caption{} \label{figure A - N[v]}
\end{figure}

The following proposition is a key to the whole proof. This proposition allows us to determine the homotopy types of $I(A_n)$ and $I(B_n)$ inductively.

\begin{proposition} \label{proposition key}
The following hold:
\begin{enumerate}[$(1)$]
\item If $n = 2k + 1$, then $I(A_n)$ is homotopy equivalent to a wedge of spheres whose dimension is at most $3k$. If $n = 2k$, then $I(A_n)$ is homotopy equivalent to a wedge of spheres whose dimension is at most $3k-1$.

\item For $n \ge 4$, the inclusion $I(A_n - N[v_n]) \hookrightarrow I(A_n - v_n)$ is null-homotopic. In particular, we have
\[ I(A_n) \simeq I(X_n) \vee \Sigma^4 I(B_{n-3}).\]

\item If $n = 2k + 1$, then $I(B_n)$ is homotopy equivalent to a wedge of spheres whose dimension is at most $3k+1$. If $n = 2k$, then $I(B_n)$ is homotopy equivalent to a wedge of spheres whose dimension is at most $3k-1$.

\item For $n \ge 5$, the inclusion $I(B_n - N[v_n]) \hookrightarrow I(B_n - v_n)$ is null-homotopic. In particular, we have
\[ I(B_n) \simeq I(Y_n) \vee \Sigma^6 I(A_{n-4}).\]
\end{enumerate}
\end{proposition}
\begin{proof}
We simultaneously show these four statements by induction on $n$. The case $n \le 3$ immediately follows from Lemma \ref{lemma small AB}. Suppose $n \ge 4$. We note that in the statement of (2) (or (4)), Theorem \ref{thm cofiber} and Lemma \ref{lemma A - B} imply that the former assertion implies the latter.

We now show (1) and (2). Suppose that $n$ is odd, and set $n = 2k + 1$. Then Lemma \ref{lemma A - B} implies that
\[ I(A_{2k+1} - v_{2k+1}) = I(X_{2k+1}) \simeq {\rm pt}.\]
Hence (2) is clear and $I(A_{2k+1}) \simeq \Sigma I(A_{2k+1} - N[v_{2k+1}]) \simeq \Sigma^4 I(B_{2k-2})$. Then, by the inductive hypothesis, $I(B_{2k-2})$ is homotopy equivalent to a wedge of spheres whose dimension is at most $3k -4$. This implies (1) in this case.

Next suppose that $n$ is even, and set $n = 2k$. Lemmas \ref{lemma X} and \ref{lemma A - B} imply
\[ I(A_{2k} - v_{2k}) = I(X_{2k}) \simeq S^{3k-1} \quad \textrm{and} \quad I(A_{2k} - N[v_{2k}]) \simeq \Sigma^3 I(B_{2k-3}).\]
By the inductive hypothesis, $I(B_{2k-3})$ is homotopy equivalent to a wedge of spheres whose dimension is at most $3 (k-2) + 1 = 3k - 5$. Hence the inclusion $\Sigma^3 I(B_{2k - 3}) \to I(X_{2k})$ is null-homotopic. This implies (2) and
\[ I(A_{2k}) \simeq I(X_{2k}) \vee \Sigma^4 I(B_{2k-3}).\]
This means $I(A_{2k})$ is homotopy equivalent to a wedge of spheres whose dimension is at most $3k -1$. This completes the proof of (1).

Next we show (3) and (4). Suppose that $n$ is odd, and set $n = 2k+1$. Then Lemmas \ref{lemma Y} and \ref{lemma A - B} imply
\[ I(B_{2k+1} - v_{2k+1}) = I(Y_{2k+1}) \simeq S^{3k+1} \quad \textrm{and} \quad I(B_{2k+1} - N[v_{2k+1}]) \simeq \Sigma^5 I(A_{2k-3}).\]
By the inductive hypothesis, $I(A_{2k-3}) = I(A_{2(k-2) + 1})$ is homotopy equivalent to a wedge of spheres whose dimension is at most $3(k-2) = 3k - 6$. Thus the inclusion $\Sigma^5 I(A_{2k-3}) \to I(X_{2k+1})$ is null-homotopic, which implies (4), and
\[ I(B_{2k+1}) \simeq I(Y_{2k+1}) \vee \Sigma^6 I(A_{2k-3}).\]
This implies that $I(B_{2k+1})$ is homotopy equivalent to a wedge of spheres whose dimension is at most $3k+1$, which implies (3).

Finally suppose that $n$ is even, and set $n = 2k$. We have already showed the case that $k = 2$ in the previous lemma and its proof. Suppose $k \ge 3$. Then Lemmas \ref{lemma Y} and \ref{lemma A - B} imply
\[ I(B_{2k} - v_{2k}) = I(Y_{2k}) \simeq S^{3k-1} \quad \textrm{and} \quad I(B_{2k} - N[v_{2k}]) \simeq \Sigma^5 I(A_{2k - 4}).\]
Then $I(A_{2k - 4}) = I(A_{2(k-2)})$ is homotopy equivalent to a wedge of spheres whose dimension is at most $3(k-2) - 1 = 3k - 7$. This implies that the inclusion $\Sigma^5 I(A_{2k-4}) \to I(Y_{2k})$ is null-homotopic, which implies (4), and there is a homotopy equivalence
\[ I(B_{2k}) \simeq I(X_{2k}) \vee \Sigma^6 I(A_{2k - 4}).\]
This implies that $I(B_{2k})$ is homotopy equivalent to a wedge of spheres whose dimension is at most $3k-1$. This completes the proof.
\end{proof}

\begin{figure}
\begin{picture}(350,140)(0,-20)
\multiput(30,0)(0,20){6}{\circle*{3}}
\multiput(50,0)(0,20){6}{\circle*{3}}
\multiput(70,0)(0,20){6}{\circle*{3}}
\multiput(90,0)(0,20){6}{\circle*{3}}
\multiput(110,0)(0,20){2}{\circle*{3}}
\multiput(110,60)(0,20){3}{\circle*{3}}
\multiput(130,80)(0,20){2}{\circle*{3}}
\put(130,0){\circle*{3}}

\put(110,40){\circle{3}}
\multiput(130,20)(0,20){2}{\circle{3}}

\multiput(30,0)(20,0){4}{\line(0,1){100}}
\put(20,0){\line(1,0){110}}
\multiput(20,80)(0,20){2}{\line(1,0){110}}
\multiput(20,20)(0,40){2}{\line(1,0){90}}
\multiput(110,0)(20,80){2}{\line(0,1){20}}
\put(110,60){\line(0,1){40}}
\put(20,40){\line(1,0){70}}

\put(135,37){$v_n$}
\put(0,47){$\cdots$}
\put(40,-20){$I(B_n - N[v_n])$}

\multiput(210,0)(0,20){6}{\circle*{3}}
\put(230,0){\circle*{3}}
\multiput(230,40)(0,20){3}{\circle*{3}}
\multiput(250,20)(20,0){2}{\circle*{3}}
\multiput(250,100)(20,0){2}{\circle*{3}}
\multiput(270,60)(20,0){2}{\circle*{3}}
\multiput(290,0)(20,0){2}{\circle*{3}}
\multiput(310,80)(0,20){2}{\circle*{3}}

\multiput(230,20)(0,80){2}{\circle{3}}
\multiput(250,0)(20,0){2}{\circle{3}}
\multiput(250,40)(0,20){3}{\circle{3}}
\multiput(270,40)(0,40){2}{\circle{3}}
\put(290,20){\circle{3}}
\multiput(290,80)(0,20){2}{\circle{3}}

\put(210,0){\line(0,1){100}}
\put(210,40){\line(0,1){40}}
\put(230,40){\line(0,1){40}}
\multiput(200,0)(0,40){3}{\line(1,0){30}}
\put(200,60){\line(1,0){30}}
\multiput(200,20)(0,80){2}{\line(1,0){10}}
\multiput(250,20)(0,80){2}{\line(1,0){20}}
\put(270,60){\line(1,0){20}}
\put(290,0){\line(1,0){20}}
\put(310,80){\line(0,1){20}}

\put(160,47){$\simeq$}
\put(180,47){$\cdots$}
\put(230,-20){$\Sigma^5 I(A_{n-4})$}

\end{picture}
\caption{} \label{figure B - N[v]}
\end{figure}

Combining (2) and (4) of Proposition \ref{proposition key}, we have the following.

\begin{corollary} \label{corollary key}
For $n \ge 8$, there is a following homotopy equivalence:
\[ I(A_n) \simeq I(X_n) \vee \Sigma^4 I(Y_{n-3}) \vee \Sigma^{10} I(A_{n-7}).\]
\end{corollary}

Combining Lemma \ref{lemma small AB} and (2) of Proposition \ref{proposition key}, we have the following.

\begin{corollary} \label{corollary small A}
There are following homotopy equivalences:
\[ I(A_1) \simeq {\rm pt}, I(A_2) \simeq S^2, I(A_3) \simeq S^3, I(A_4) \simeq S^5 \vee S^5,\]
\[ I(A_5) \simeq S^6, I(A_6) \simeq S^8 \vee S^8, I(A_7) \simeq S^9 \vee S^9.\]
\end{corollary}

In the rest of this subsection, we prove the following proposition.

\begin{figure}
\begin{picture}(350,140)(0,-20)
\multiput(30,0)(0,20){6}{\circle*{3}}
\multiput(50,0)(0,20){6}{\circle*{3}}
\multiput(70,0)(0,20){6}{\circle*{3}}
\multiput(90,0)(0,20){6}{\circle*{3}}
\multiput(110,0)(0,20){3}{\circle*{3}}
\multiput(110,80)(0,20){2}{\circle*{3}}
\multiput(130,0)(0,20){2}{\circle*{3}}
\put(130,100){\circle*{3}}

\put(110,60){\circle{3}}
\multiput(130,40)(0,20){3}{\circle{3}}

\multiput(30,0)(20,0){4}{\line(0,1){100}}
\multiput(20,0)(0,20){2}{\line(1,0){110}}
\multiput(20,40)(0,40){2}{\line(1,0){90}}
\put(20,60){\line(1,0){70}}
\put(20,100){\line(1,0){110}}
\put(110,0){\line(0,1){40}}
\put(130,0){\line(0,1){20}}
\put(110,80){\line(0,1){20}}

\put(137,57){$w_n$}
\put(0,47){$\cdots$}
\put(40,-20){$I(\Gamma_n - N[w_n])$}

\multiput(210,0)(0,20){6}{\circle*{3}}
\multiput(230,0)(0,20){6}{\circle*{3}}
\multiput(250,0)(0,20){6}{\circle*{3}}
\multiput(270,0)(0,20){6}{\circle*{3}}
\multiput(290,0)(0,20){2}{\circle*{3}}
\multiput(290,60)(0,20){3}{\circle*{3}}
\multiput(310,80)(0,20){2}{\circle*{3}}
\put(310,0){\circle*{3}}

\put(290,40){\circle{3}}
\multiput(310,20)(0,20){3}{\circle{3}}

\multiput(210,0)(20,0){4}{\line(0,1){100}}
\put(200,0){\line(1,0){110}}
\multiput(200,80)(0,20){2}{\line(1,0){110}}
\multiput(200,20)(0,40){2}{\line(1,0){90}}
\multiput(290,0)(20,80){2}{\line(0,1){20}}
\put(290,60){\line(0,1){40}}
\put(200,40){\line(1,0){70}}

\put(315,37){$v_n$}
\put(180,47){$\cdots$}
\put(220,-20){$I(\Gamma_n - N[v_n])$}

\put(160,47){$\cong$}
\put(180,47){$\cdots$}
%\put(230,-20){$\Sigma^5 I(A_{n-4})$}

\end{picture}
\caption{} \label{figure B - N[w] = B - N[v]}
\end{figure}

\begin{proposition} \label{proposition key 2}
For $n \ge 5$, we have the following homotopy equivalence:
\[ I(\Gamma_n) \simeq I(Y_n) \vee \bigvee_2 \Sigma^6 I(A_{n-4}).\]
\end{proposition}
\begin{proof}
Set $w_n = (n,4)$. We first show that the inclusion $I(\Gamma_n - N[w_n]) \hookrightarrow I(\Gamma_n - w_n)$ is null-homotopic. Note that the inclusion $I(\Gamma_n - N[w_n]) \hookrightarrow I(\Gamma_n - w_n)$ has the following factorization:
\[ I(\Gamma_n - N[w_n]) \to I(Y_n) \to I(\Gamma_n - w_n).\]
Hence we show that the inclusion $I(\Gamma_n - N[w_n]) \hookrightarrow I(Y_n)$ is null-homotopic.

Let $\alpha$ be the automorphism of $\Gamma_n$ sending a vertex $(i,j)$ to the vertex $(i, 7 - j)$. Then $\alpha$ restricts to isomorphisms $B_n - N[w_n]\xrightarrow{\cong}B_n - N[v_n]$ (see Figure \ref{figure B - N[w] = B - N[v]}) and $Y_n\xrightarrow{\cong}Y_n$. Hence it suffices to see that the inclusion $I(\Gamma_n - N[v_n])\hookrightarrow I(Y_n)$ is null-homotopic. Since
%Then we have the following commutative diagram
%$$
%\xymatrix{
%I(\Gamma_n - N[w_n]) \ar[r] \ar[d]_{\alpha} & I(Y_n) \ar[d]^{\alpha}\\
%I(\Gamma_n - N[v_n]) \ar[r] & I(Y_n).
%}
%$$
%Recall that $v_n$ denotes the vertex $(n,3)$. Here, the horizontal arrows are inclusions. The vertical arrows are isomorphisms. It follows from (4) of Proposition \ref{proposition key} that the lower horizontal arrow
\[ I(B_n - N[v_n]) = I(\Gamma_n - N[v_n]) \hookrightarrow I(Y_n) = I(B_n - v_n),\]
Proposition \ref{proposition key} implies that the inclusion $I(\Gamma_n - N[v_n]) \hookrightarrow I(Y_n)$ is null-homotopic. This concludes that the inclusion $I(\Gamma_n - N[w_n]) \hookrightarrow I(Y_n)$ is null-homotopic, as desired.
%This concludes that the upper horizontal arrow $I(\Gamma_n - N[w_n]) \hookrightarrow I(Y_n)$ is null-homotopic, and hence that the inclusion $I(\Gamma_n - N[w_n]) \hookrightarrow I(\Gamma_n - w_n)$ is null-homotopic.

Theorem \ref{thm cofiber} implies that $I(\Gamma_n) \simeq I(\Gamma_n - w_n) \vee \Sigma I(\Gamma_n -N[w_n])$. Lemma \ref{lemma A - B} implies that $I(\Gamma_n - N[w_n]) \cong I(\Gamma_n - N[v_n]) \simeq \Sigma^5 I(A_{n-4})$. Since $\Gamma_n - w_n = B_n$, (4) of Proposition \ref{proposition key} implies
\[ I(\Gamma_n) \simeq I(\Gamma_n - w_n) \vee \Sigma I(\Gamma_n - N[w_n]) \simeq I(Y_n) \vee \bigvee_2 \Sigma^6 I(A_{n-4}).\]
This completes the proof.
\end{proof}

\subsection{Proof of Theorem \ref{main theorem}}

The goal of this subsection is to complete the proof of Theorem \ref{main theorem}, which determines the homotopy type of $I(\Gamma_n)$. By Proposition \ref{proposition key 2}, the homotopy type of $I(\Gamma_n)$ is determined by the homotopy types of $I(X_n)$, $I(Y_n)$, and $I(A_n)$. The homotopy types of $I(X_n)$ and $I(Y_n)$ have been already determined in Lemmas \ref{lemma X} and \ref{lemma Y}. Thus the next task is to determine the homotopy type of $I(A_n)$.

\begin{lemma} \label{lemma inductive}
The following hold:
\begin{enumerate}[$(1)$]
\item If $2k + 1 \ge 15$, there is a following homotopy equivalence:
\[ I(A_{2k + 1}) \simeq \Big( \bigvee_3 S^{3k} \Big) \vee \Sigma^{20} I(A_{2k+1 - 14}).\]

\item There is a following homotopy equivalence:
\[ I(A_{14m + 2k + 1}) \simeq \bigvee_3 S^{21m + 3k} \vee \cdots \vee \bigvee_3 S^{20m + 3k + 1} \vee \Sigma^{20m} I(A_{2k+1})\]
\end{enumerate}
\end{lemma}
\begin{proof}
It follows from Lemmas \ref{lemma X}, \ref{lemma Y} and Corollary \ref{corollary key} that
\begin{eqnarray*}
I(A_{2k+1}) & \simeq & I(X_{2k+1}) \vee \Sigma^4 I(Y_{2k-2}) \vee \Sigma^{10} I(A_{2k-6}) \\
& \simeq & \Sigma^4 I(Y_{2k-2}) \vee \Sigma^{10} I(X_{2k-6}) \vee \Sigma^{14} I(Y_{2k-9}) \vee \Sigma^{20} I(A_{2k-13})\\
& \simeq & S^{3k} \vee S^{3k} \vee S^{3k} \vee \Sigma^{20} I(A_{2k+1 - 14})
\end{eqnarray*}
if $2k+1 \ge 15$. We deduce (2) by iterating (1).
\end{proof}

We first determine the homotopy type of $I(A_n)$ for odd $n$.

%\begin{proposition} \label{proposition A odd}
%For $m \ge 0$, there are following homotopy equivalences:
%\[ I(A_{14m + 1}) \simeq \bigvee_3 S^{21m} \vee \cdots \vee \bigvee_3 S^{20m + 1},\]
%\[ I(A_{14m + 3}) \simeq \Big( \bigvee_3 S^{21m + 3} \vee \cdots \vee \bigvee_3 S^{20m +4} \Big) \vee S^{20m + 3},\]
%\[ I(A_{14m + 5}) \simeq \Big( \bigvee_3 S^{21m + 6} \vee \cdots \vee \bigvee_3 S^{20m + 7} \Big) \vee S^{20m + 6},\]
%\[ I(A_{14m + 7}) \simeq \Big( \bigvee_3 S^{21m + 9} \vee \cdots \vee \bigvee_3 S^{20m + 10} \Big) \vee \bigvee_2 S^{20m + 9},\]
%\[ I(A_{14m + 9}) \simeq \Big( \bigvee_3 S^{21m + 12} \vee \cdots \vee \bigvee_3 S^{20m + 13} \Big) \vee \bigvee_2 S^{20m + 12},\]
%\[ I(A_{14m + 11}) \simeq \bigvee_3 S^{21m + 15} \vee \cdots \vee \bigvee_3 S^{20m + 15},\]
%\[ I(A_{14m + 13}) \simeq \bigvee_3 S^{21m +18} \vee \cdots \vee \bigvee_3 S^{20m + 18}.\]
%\end{proposition}
\begin{proposition} \label{proposition A odd}
  For $m \ge 0$ and $0\le k\le 6$, there is a following homotopy equivalence:
  \[
    I(A_{14m+2k+1})\simeq
    \Bigl(\bigvee_{i=20m+3k+1}^{21m+3k}\bigvee_3S^i\Bigr)
    \vee\bigvee_aS^{20m+3k},
  \]
  where $a$ is a number defined by Table \ref{table:abValues}.
\end{proposition}
\begin{proof}
The case that $14m + 2k + 1 \le 7$ has been already proved in Corollary \ref{corollary small A}.
If $14m + 2k+1 > 7$, Corollary \ref{corollary key} implies
\[ I(A_{2k+1}) \simeq I(X_{2k+1}) \vee \Sigma^4 I(Y_{2k-2}) \vee \Sigma^{10} I(A_{2k - 6}) \simeq S^{3k} \vee \Sigma^{10} I(A_{2k - 6}).\]
Hence Corollary \ref{corollary small A} implies
\[ I(A_9) \simeq S^{12} \vee \Sigma^{10} I(A_2) \simeq S^{12} \vee S^{12},\]
\[ I(A_{11}) \simeq S^{15} \vee \Sigma^{10} I(A_4) \simeq S^{15} \vee S^{15} \vee S^{15},\]
\[ I(A_{13}) \simeq S^{18} \vee \Sigma^{10} I(A_6) \simeq S^{18} \vee S^{18} \vee S^{18}.\]
Then Lemma \ref{lemma inductive} completes the proof.
\end{proof}

\begin{table}[t]
  \begin{tabular}{cccccccc}
    $n$   & 0 & 1 & 2 & 3 & 4 & 5 & 6 \\
    \hline
    $a$ & 0 & 1 & 1 & 2 & 2 & 3 & 3 \\
    $b$ & - & - & 0 & 0 & 0 & 1 & 1
  \end{tabular}
  \caption{}\label{table:abValues}
\end{table}

Next we determine the homotopy type of $I(A_n)$ for even $n$.

%\begin{proposition} \label{proposition A even}
%There are following homotopy equivalences:
%\[ I(A_{14m}) \simeq \bigvee_2 S^{21m - 1} \vee \Big( \bigvee_3 S^{21m - 2} \vee \cdots \vee \bigvee_3 S^{20m} \Big) \vee \bigvee_2 S^{20m - 1},\]
%\[ I(A_{14m + 2}) \simeq \begin{cases}
%S^2 & (m = 0) \\
%\bigvee_2 S^{21m + 2} \vee \Big( \bigvee_3 S^{21m + 1} \vee \cdots \vee \bigvee_3 S^{20m + 3} \Big) \vee \bigvee_2 S^{20m + 2}, & {\rm (otherwise)}
%\end{cases}\]
%\[ I(A_{14m + 4}) \simeq \bigvee_2 S^{21m + 5} \vee \Big( \bigvee_3 S^{21m + 4} \vee \cdots \vee \bigvee_3 S^{20m + 5} \Big),\]
%\[ I(A_{14m + 6}) \simeq \bigvee_2 S^{21m + 8} \vee \Big( \bigvee_3 S^{21m + 7} \vee \cdots \vee \bigvee_3 S^{20m + 8}\Big),\]
%\[ I(A_{14m + 8}) \simeq \bigvee_2 S^{21m + 11} \vee \Big( \bigvee_3 S^{21m + 10} \vee \cdots \vee \bigvee_3 S^{20m + 11} \Big),\]
%\[ I(A_{14m + 10}) \simeq \bigvee_2 S^{21m + 14} \vee \Big( \bigvee_3 S^{21m + 13} \vee \cdots \vee \bigvee_3 S^{20m + 14} \Big) \vee S^{20m + 13},\]
%\[ I(A_{14m + 12}) \simeq \bigvee_2 S^{21m + 17} \vee \Big( \bigvee_3 S^{21m + 16} \vee \cdots \vee \bigvee_3 S^{20m + 17} \Big) \vee S^{20m + 16},\]
%\end{proposition}
\begin{proposition} \label{proposition A even}
  Assume $m \ge 0$ and $0\le k\le 6$.
  Except for $I(A_2)\simeq S^2$,
  there is a following homotopy equivalence:
  \[
    I(A_{14m+2k})\simeq
    \begin{cases}
      \displaystyle
      \bigvee_2S^{21m+3k-1}\vee
      \Bigl(\bigvee_{i=20m+3k}^{21m+3k-2}\bigvee_3S^i\Bigr)
      \vee\bigvee_2S^{20m+3k-1} & (k = 0, 1) \\
      \displaystyle
      \bigvee_2S^{21m+3k-1}\vee
      \Bigl(\bigvee_{i=20m+3k-1}^{21m+3k-2}\bigvee_3S^i\Bigr)
      \vee\bigvee_bS^{20m+3k-2} & (2\le k\le 6),
    \end{cases}
  \]
  where $b$ is a number defined by Table \ref{table:abValues}.
\end{proposition}
\begin{proof}
The case $14m + 2k \le 7$ follows from Corollary \ref{corollary small A}. Suppose that $14m + 2k > 7$. Then Corollary \ref{corollary key} implies
\begin{eqnarray*}
I(A_{14m + 2k}) & \simeq & I(X_{14m + 2k}) \vee \Sigma^4 I(Y_{14m + 2k - 3}) \vee \Sigma^{10} I(A_{14m + 2k - 7}) \\
& \simeq & \bigvee_2 S^{21m + 3k - 1} \vee \Sigma^{10} I(A_{14m + 2k - 7}).
\end{eqnarray*}
Then Proposition \ref{proposition A odd} completes the proof.
\end{proof}

Now we complete the proof of our main theorem.

\begin{proof}[Proof of Theorem \ref{main theorem}]
Suppose $n \ge 5$. Then Proposition \ref{proposition key 2} asserts
\[ I(\Gamma_n) \simeq I(Y_n) \vee \bigvee_2 \Sigma^6 I(A_{n-4}).\]
Then the proof is completed by the combination of Propositions \ref{proposition A odd} and \ref{proposition A even}, and Lemma \ref{lemma Y}. The case $n < 5$ follows from the known results (see \cite{Kozlov}, \cite{Adamaszek2} and \cite{MW}).
\end{proof}

\section*{Acknowledgment}
The authors are grateful to the anonymous referee for useful comments. The first author and second author were partially supported by JSPS KAKENHI Grant Number JP19K14536 and JP20J00404, respectively.


\begin{thebibliography}{99}
\bibitem[Ad1]{Adamaszek1} M. Adamaszek, \emph{Splittings of independence complexes and the powers of cycles}, J. Comb. Theory Ser. A, 119 (2012) 1031-1047.

\bibitem[Ad2]{Adamaszek2} M. Adamaszek, \emph{Hard squares on cylinders revisited}, arXiv:1202.1655

%\bibitem[AGS]{AGS} A. Singh, S. Goyal, S. Shukla, {\it Matching complexes of $\mathbf{3 \times n}$-grid graphs}, Electr. J. Combin.  (4) 28 \# P4.16, 2021.

\bibitem[Ba]{Barmak} J. A. Barmak, \emph{Star clusters in independence complexes of graphs}, Adv. Math. 241 (2013) 33-57.

%\bibitem[BGJM]{BGJM} M. Bayer, B. Goeckner, M. Jeli\'c Milutinovi\'c, \emph{Manifold matching complexes}, Mathematika {\bf 66} (2020) 973-1002.

\bibitem[Br]{Braun} B. Braun, \emph{Independence complexes of stable Kneser graphs}, Electr. J. Comb., 18(1), 2011.

\bibitem[BH]{BH} B. Braun, W. K. Hough. \emph{Matching and independence complexes related to small grids}, Electr. J. Comb., {\bf 24}(4), 2017.

\bibitem[BLN]{BLN} M. Bousquet-M\'elou, S. Linusson, E. Nevo, \emph{On the independence complex of square grids}, J. Algebr. Comb. 2008, 27: 423-450

\bibitem[EH]{EH} R. Ehrenborg, G. Hetyei, \emph{The topology of independence complex}, Eur. J. Comb. {\bf 27} (2006) 906-923.

\bibitem[En]{Engstrom} A. Engstr\"{o}m, \emph{Complexes of directed trees and independence complexes}, Discrete Math. 309 (2009) 3299-3309.

\bibitem[FSv]{FSV} P. Fendley, K. Schoutens, and H. van Eerten, \emph{Hard squares with negative activity}, J. Phys. A: Math. Gen. {\bf 38} (2005), no. 2, 315-322.

\bibitem[GSS]{GSS} S. Goyal, S. Shukla, A. Singh, \emph{Homotopy Type of Independence Complexes of Certain Families of Graphs,} Contribution to Discrete Mathematics, vol. 16, no. 3, (2021) 74-92.

%\bibitem{Hatcher} A. Hatcher; Algebraic topology, Cambridge University Press, 2001.

\bibitem[Ir]{Iriye} K. Iriye, \emph{On the homotopy types of the independence complexes of grid graphs with cylindrical identification}, Kyoto J. Math., 52(3): 479-501 (2012)

%\bibitem{Jonsson} J. Jonsson, {\it On the topology of independence complexes of triangle-free graphs,} Preprint.

\bibitem[Jo1]{Jonsson1} J. Jonsson, \emph{Hard squares with negative activity and rhombus tilings of the plane}, Electr. J. Comb. 13 (1) (2006), \#R67

\bibitem[Jo2]{Jonsson2} J. Jonsson, \emph{Hard squares with negative activity on cylinders with odd circumference}, Electr. J. Comb. 16 (2) (2009), \#R5

\bibitem[Jo3]{Jonsson3} J. Jonsson, \emph{Certain homology cycles of the independence complex of grids}, Discrete \& Computational Geometry 43(4): 927-950 (2010)

\bibitem[Ka]{Kawamura}  K. Kawamura, \emph{Independence complexes of chordal graphs}, Discrete Math. {\bf 310} (2010) 2204-2211.

\bibitem[Ko1]{Kozlov} D.N. Kozlov, \emph{Complexes of directed trees}, J. Comb. Theory Ser. A, 88 (1999) 112-122.

\bibitem[Ko2]{Kozlov book} D.N. Kozlov, Combinatorial algebraic topology, Springer, Berlin, Algorithms and Computation in Mathematics, Vol. 21, 2008.

\bibitem[MT]{MT} M. Marietti, D. Testa, {\it A uniform approach to complexes arising from forests,} Electr. J. Comb. {\bf 15} (2008).

%\bibitem{MT} M. Marietti, D. Testa; {\it Cores of simplicial complexes,} Discrete Comput. Geom. {\bf 40} (2008), 444-468.

%\bibitem{Matousek} J. Matou\v{s}ek; Using the Borsuk-Ulam Theorem, Lectures on Topological Methods in Combinatorics and Geometry, Springer, Universitext, second edition, 2007.

\bibitem[Ma1]{Matsushita1} T. Matsushita, \emph{Matching complexes of small grids}, Electr. J. Combin. Volume 26, Issue 3, (July, 2019)

\bibitem[Ma2]{Matsushita2} T. Matsushita, \emph{Matching complexes of polygonal line tilings}, Hokkaido Mathematical Journal, Vol. 51 (2022) 339-359.

\bibitem[MW]{MW} T. Matsushita, S. Wakatsuki, \emph{Independence complexes of $(n \times 4)$ and $(n \times 5)$-grid graphs}, arXiv:2203.16391

%\bibitem[M]{Meshulam} R. Meshulam; {\it Domination numbers and homology}, J. Comb. Theory Ser. A, 102 (2003) 321-330.

%\bibitem{Taylan} D. Taylan; {\it Matching trees for simplicial complexes and homotopy type of devoid complexes of graphs,} Order, {\bf 33} (2016), 459-476.

%\bibitem[NR]{NR} U. Nagel, V. Reiner; {\it Betti numbers of monomial ideals and shifted skew shapes}, Electr. J. Comb. {\bf 16} (2009).

\bibitem[Ok]{Okura} K. Okura, \emph{Simple homotopy types of independence complexes of graphs involving grid graphs}, arXiv:1908.09356

%\bibitem[O2]{Okura2} K. Okura, {\it Independence complex of the lexicographic product of a forest}, arXiv:2109.04181

\bibitem[PRSZ]{PRSZ} L. Polterovich, D. Rosen, K. Samvelyan, J. Zhang, {Topological Persistence in Geometry and Analysis}, University Lecture Series, American Mathematical Society, 2020,

%\bibitem[S]{Singh} A. Singh, \emph{The topology of independence complexes of square grids},

\bibitem[Th]{Thapper}  J. Thapper, \emph{Independence complexes of cylinders constructed from square and hexagonal grid graphs}, arXiv:0812.1165

%\bibitem[Wac]{Wachs} M. L. Wachs. \emph{Topology of matching, chessboard, and general bounded degree graph complexes}, Algebra Universalis, {\bf 49} (4):345-385, 2003.
\end{thebibliography}
\end{document}